\documentclass[12pt,reqno]{amsart}
\usepackage{amssymb}
\usepackage[latin1]{inputenc}
\usepackage{dsfont}
\usepackage{hyperref}
\usepackage{fullpage}
\newtheorem{theorem}{Theorem}[section]
\newtheorem{lemma}[theorem]{Lemma}
\newtheorem{proposition}[theorem]{Proposition}
\newtheorem{corollary}[theorem]{Corollary}
\theoremstyle{definition}
\newtheorem{definition}[theorem]{Definition}

\theoremstyle{remark}
\newtheorem{remark}[theorem]{Remark}
\numberwithin{equation}{section}
\begin{document}
\title{On some universal $\sigma$-finite measures and some extensions of Doob's optional stopping theorem}
\author[J. Najnudel]{Joseph Najnudel}
\address{Institut f\"ur Mathematik, Universit\"at Z\"urich, Winterthurerstrasse 190,
8057-Z\"urich, Switzerland}
\email{\href{mailto:joseph.najnudel@math.uzh.ch}{joseph.najnudel@math.uzh.ch}}
\author[A. Nikeghbali]{Ashkan Nikeghbali}
\email{\href{mailto:ashkan.nikeghbali@math.uzh.ch}{ashkan.nikeghbali@math.uzh.ch}}
 \dedicatory{Dedicated
to Marc Yor for his 60th birthday}

\date{\today}
\begin{abstract}
In this paper, we associate, to any submartingale of class $(\Sigma)$, defined on a filtered probability
space $(\Omega, \mathcal{F}, \mathbb{P}, (\mathcal{F}_t)_{t \geq 0})$, which satisfies some technical conditions, 
a $\sigma$-finite measure $\mathcal{Q}$ on $(\Omega, \mathcal{F})$, such that for all 
$t \geq 0$, and for all events $\Lambda_t
 \in \mathcal{F}_t$: $$ \mathcal{Q} [\Lambda_t, g\leq t] = \mathbb{E}_{\mathbb{P}} [\mathds{1}_{\Lambda_t} X_t]$$
where $g$ is the last hitting time of zero of the process $X$. This measure $\mathcal{Q}$ has already been 
defined in several particular cases, some of them are involved in the study of Brownian penalisation, 
and others are related with problems in mathematical finance. More precisely, the existence of $\mathcal{Q}$ 
in the general case solves a problem stated by D. Madan, B. Roynette and M. Yor, in a paper studying the link
between Black-Scholes formula and last passage times of certain submartingales. 
Moreover, the equality defining $\mathcal{Q}$ remains true if one replaces the fixed time $t$ by any bounded
stopping time. This generalization can be viewed as an extension of Doob's optional stopping theorem. 
\end{abstract}
\maketitle
\section{Introduction}
This work finds its origin in a recent paper by Madan, Roynete and Yor \cite{MRY} and a set of lectures by Yor \cite{BeYor} where the authors are able to represent the price of a European put option in terms of the probability distribution of some last passage time. More precisely, they prove that if $(M_t)_{t \geq 0}$ is a continuous nonnegative local martingale  defined on a filtered probability space $(\Omega,\mathcal{F},(\mathcal{F}_t)_{t \geq 0},\mathbb{P})$ satisfying the usual assumptions, and such that  $\lim_{t\to\infty}M_t=0$, then 
\begin{equation}\label{BS}
	(K-M_t)^+=K\mathbb{P}(g_K\leq t|\mathcal{F}_t)
	\end{equation}where $K\geq0$ is a constant and $g_K=\sup\{t \geq 0: M_t=K\}$.  Formula (\ref{BS}) tells that it is enough to know the terminal value of the submartingale $(K-M_t)^+$ and its last zero $g_K$ to reconstruct it.  Yet a nicer interpretation of  (\ref{BS}) is suggested in \cite{BeYor} and \cite{MRY}:  there exists a measure $\mathcal{Q}$, a random time $g$, such that the submartingale $X_t=(K-M_t)^+$ satisfies
\begin{equation}\label{masterequation}
\mathcal{Q} \left[ F_t \, \mathds{1}_{g \leq t} \right] = \mathbb{E} \left[F_t X_t \right],	
\end{equation}
for any $t\geq0$ and for any  bounded $\mathcal{F}_t$-measurable random variable $F_t$. Indeed it easily follows from (\ref{BS}) that in this case  $\mathcal{Q}=K. \mathbb{P}$ and $g=g_K$.  It is also clear that if a stochastic process $X$ satisfies  (\ref{masterequation}), then it is a submartingale. The problem of finding the class of submartingales which satisfy (\ref{masterequation}) is posed in  \cite{BeYor} and \cite{MRY} and is the main motivation of this paper:
\newline
\textbf{Problem 1 (\cite{BeYor} and \cite{MRY}):}  for which submartingales $X$ can we find a $\sigma$-finite measure $\mathcal{Q}$ and the end of an optional set $g$ such that
\begin{equation}\label{masterequation2}
\mathcal{Q} \left[ F_t \, \mathds{1}_{g \leq t} \right] = \mathbb{E} \left[F_t X_t \right]?
\end{equation}
\newline
Identity (\ref{masterequation2}) is reminiscent of the stopping theorem for uniformly integrable martingales. Indeed, if $M$ is a c\`adl\`ag, uniformly integrable martingale, then it can be represented as $M_t=\mathbb{E}[M_\infty|\mathcal{F}_t]$, and hence the terminal value of $M$, i.e. $M_\infty$, is enough to obtain the martingale $M$. But we also note that if we write $g=\sup\{t \geq 0 :\;M_t=0\}$, then $$M_t=\mathbb{E}[M_\infty \mathds{1}_{g\leq t}|\mathcal{F}_t],$$since $\mathbb{E}[M_\infty \mathds{1}_{g> t}|\mathcal{F}_t]=0$. Thus (\ref{masterequation2}) holds for $M$, where the measure $\mathcal{Q}$ is the signed measure $\mathcal{Q}=M_\infty.\mathbb{P}$. Consequently, the stopping theorem can also be interpreted as the existence of a (signed) measure and the end of an optional set from which one can recover the uniformly integrable martingale $M$. But (\ref{masterequation2}) does not admit a straightforward generalization to martingales which are not uniformly integrable: indeed, such a   measure $\mathcal{Q}$ would be  real valued and infinite. We hence propose the following problem:
\newline
\textbf{Problem 2:} given a continuous martingale $M$, can we find two $\sigma$-finite measures $\mathcal{Q}^{(+)}$ and $\mathcal{Q}^{(-)}$, such that for all $t\geq0$ and for all bounded $\mathcal{F}_t$-measurable variables $F_t$:
\begin{equation}\label{masterequation3}
\left(\mathcal{Q}^{(+)}-\mathcal{Q}^{(-)}\right) \left[ F_t \, \mathds{1}_{g \leq t} \right] = \mathbb{E} \left[F_t M_t \right],
\end{equation}
with $g=\sup\{t \geq 0:\;M_t=0\}$?
\newline
Identities (\ref{masterequation2}) and (\ref{masterequation3}) can hence be interpreted as an extension of  Doob's optional stopping theorem for fixed times $t$. 
\newline
It is also noticed in \cite{MRY} that other instances of formula (\ref{masterequation}) have already been discovered: for example, in \cite{AY1}, Az\'ema and Yor proved that for any continuous and uniformly martingale $M$, (\ref{masterequation2}) holds for $X_t=|M_t|$, $\mathcal{Q}=|M_\infty|.\mathbb{P}$ and $g=\sup\{t \geq 0: M_t=0\}$, or equivalently $$|M_t|=\mathbb{E}[|M_\infty|\mathds{1}_{g\leq t}|\mathcal{F}_t].$$ Here again the measure $\mathcal{Q}$ is finite. Recently,  other particular cases where the measure $\mathcal{Q}$ is not finite  were obtained by  Najnudel, Roynette and Yor
in their study of Brownian penalisation (see \cite{NRY}). For example, they prove the existence of the measure $\mathcal{Q}$ when $X_t=|W_t|$ is the absolute value of the standard Brownian Motion. In this case, the measure $\mathcal{Q}$ is not finite but only $\sigma$-finite and is singular with respect to the Wiener measure: it satisfies $\mathcal{Q}(g=\infty)=0$, where $g=\sup\{t \geq 0:\;W_t=0\}$. 

In the special case where the submartingale $X$ is of class $(D)$, Problem 1 was recently solved\footnote{In fact, as mentioned in \cite{CNP}, the solution is essentially contained and somehow hidden in \cite{AMY}.} in \cite{CNP} in relation with the study of the draw-down process..  In this case, the measure $\mathcal{Q}$ is finite. The relevant family of submartingales is  the class $(\Sigma)$:
\begin{definition}[(\cite{N,Y})]
Let $(\Omega,\mathcal{F},(\mathcal{F}_t)_{t \geq 0},\mathbb{P})$ be a filtered probability space. A nonnegative (local) submartingale $(X_t)_{t \geq 0}$ is of class $(\Sigma)$, if it can be decomposed as
$X_t = N_t + A_t$ where $(N_t)_{t \geq 0}$ and $(A_t)_{t \geq 0}$ are $(\mathcal{F}_t)_{t \geq 0}$-adapted 
processes satisfying the following assumptions:
\begin{itemize}
\item $(N_t)_{t \geq 0}$ is a c\`adl\`ag (local) martingale;
\item $(A_t)_{t \geq 0}$ is a continuous increasing process, with $A_0 = 0$;
\item The measure $(dA_t)$ is carried by the set $\{t \geq 0, X_t = 0 \}$.
\end{itemize}
\end{definition}
The definition of the class $(\Sigma)$ goes back to Yor (\cite{Y}) when $X$ is continuous  and some of its main
 properties
which we shall use frequently in this paper were studied in \cite{N}. 
It is shown in \cite{AMY} and \cite{CNP} that if $X$ is of class $(\Sigma)$ and of class $(D)$, then it satisfies (\ref{masterequation}) with $g=\sup\{t \geq 0:\; X_t=0\}$ and $\mathcal{Q}=X_\infty. \mathbb{P}$, or equivalently $$X_t=\mathbb{E}[X_\infty\mathds{1}_{g\leq t}|\mathcal{F}_t].$$ Now, what happens if $X$ is of class $(\Sigma)$ but satisfies $A_\infty=\infty$ almost surely? This is the case for example if one works on the space $\mathcal{C}(\mathbb{R}_+,\mathbb{R})$ of continuous functions endowed with the filtration $(\mathcal{F}_t)_{t \geq 0}$ generated by the coordinate process $(Y_t)_{t \geq 0}$ and with the Wiener measure $\mathbb{W}$, and if $X_t=|Y_t|$ (here the increasing process is the local time which is known to be infinite almost surely at infinity).  In this case, as it was already mentioned,  the existence of the measure $\mathcal{Q}$, which is singular
 with respect to $\mathbb{W}$, was established in \cite{NRY}.  But the general situation  when $X$ is of class $(\Sigma)$ and satisfies $A_\infty=\infty$ almost surely is nevertheless more subtle: one will need to make some non-standard assumptions on the filtration $(\mathcal{F}_t)_{t \geq 0}$. 
Indeed,
let us assume that in the filtration $(\mathcal{F}_t)_{t \geq 0}$, $\mathcal{F}_0$ contains 
all the $\mathbb{P}$-negligible sets (i.e. the filtration is complete), and that
under $\mathbb{P}$, $A_{\infty} = \infty$ almost surely, and then $g = \infty$ a.s. (e.g.  $X$ is the absolute value of a Brownian motion). For all $t \geq 0$, the event $\{g >
 t\}$ has probability one (under $\mathbb{P}$) and then, is in $\mathcal{F}_0$ and, a fortiori, in $\mathcal{F}_t$. 
If one assumes that $\mathcal{Q}$ exists, one has
$$\mathcal{Q} [g > t, g \leq t] = \mathbb{E}_{\mathbb{P}} [\mathds{1}_{g > t} \, X_t],$$
and then:
$$\mathbb{E}_{\mathbb{P}} [A_t] \leq  \mathbb{E}_{\mathbb{P}}[X_t] =  0$$
which is absurd.
\newline
The goal of this paper is to show that under some technical conditions on the filtration $(\mathcal{F}_t)_{t \geq 0}$,   Problem 1 can  be solved for all submartingales of the class $(\Sigma)$.  The measure $\mathcal{Q}$ is constructed explicitly. Since for continuous martingales, $M^+$ and $M^-$ are of class $(\Sigma)$, we shall be able to solve Problem 2 and hence interpret our results as an extension of Doob's optional stopping theorem.  Our approach is based on martingale techniques only and we are hence able to obtain the measure $\mathcal{Q}$ for a wide range of processes which can possibly jump, thus including the generalized Az\'ema submartingales in the filtration of the zeros of Bessel processes of dimension in $(0,2)$ and the draw-down process $X_t=S_t-M_t$ where $M$ is a martingale with no positive jumps and $S_t=\sup_{u\leq t} M_u$. In particular, the existence of $\mathcal{Q}$ does not require any scaling or Markov property for $X$.
More precisely, the paper is organized as follows:
\begin{itemize}
	\item in Section 2, we state our technical condition on the filtration $(\mathcal{F}_t)_{t \geq 0}$ and then state and prove our main theorem about the existence and the uniqueness of the measure $\mathcal{Q}$ for submartingales of the class $(\Sigma)$. We then deduce the solution to Problem 2, hence interpreting (\ref{masterequation2}) and (\ref{masterequation3}) together as an extension of Doob's optional stopping theorem. We also give the image of the measure $\mathcal{Q}$ by the functional $A_\infty$;
	\item in Section 3, we give several examples of such a measure $\mathcal{Q}$ in classical and less classical settings.
\end{itemize}
\section{Description of the $\sigma$-finite measure}
The main result  of this paper states that the existence and the uniqueness of $\mathcal{Q}$ essentially
 holds for all submartingales 
$(X_t)_{t \geq 0}$ of class $(\Sigma)$. However, before stating this result, we need to introduce 
some technicalities on filtrations. 
\subsection{A measure extension theorem}
Let us first introduce the following definition:
\begin{definition} \label{P}
Let $(\Omega, \mathcal{F}, (\mathcal{F}_t)_{t \geq 0})$ be a filtered measurable space, such that
$\mathcal{F}$ is the $\sigma$-algebra generated by $\mathcal{F}_t$, $t \geq 0$: $\mathcal{F}=\bigvee_{t\geq0}\mathcal{F}_t$. We shall say that the property (P) holds if and only if $\mathcal{F}_t = \mathcal{G}_{t+}$, where $(\mathcal{G}_t)_{t \geq 0}$
is a filtration, enjoying the following conditions: 
\begin{itemize}
\item For all $t \geq 0$, $\mathcal{G}_t$ is generated by a countable number of sets;
\item For all $t \geq 0$, there exists a Polish space $\Omega_t$, and a surjective map 
 $\pi_t$ from $\Omega$ to $\Omega_t$, such that $\mathcal{G}_t$ is the $\sigma$-algebra of the inverse
 images, by $\pi_t$, of Borel sets in $\Omega_t$, and such that for all $B \in \mathcal{G}_t$, 
 $\omega \in \Omega$, $\pi_t (\omega) \in \pi_t(B)$ implies $\omega \in B$;
\item If $(\omega_n)_{n \geq 0}$ is a sequence of elements of $\Omega$, such that for all $N \geq 0$,
$$\bigcap_{n = 0}^{N} A_n (\omega_n) \neq \emptyset,$$
where $A_n (\omega_n)$ is the intersection of the sets in $\mathcal{G}_n$ containing $\omega_n$, 
then:
$$\bigcap_{n = 0}^{\infty} A_n (\omega_n) \neq \emptyset.$$
\end{itemize}
\end{definition}
\noindent
Note that if $(\Omega, \mathcal{F}, (\mathcal{F}_t)_{t \geq 0})$ satisfies the property (P), then 
$(\mathcal{F}_t)_{t \geq 0}$ is right-continuous. 
The  technical definition \ref{P} is needed in order to state the following lemma, which is fundamental
 in the construction of the measure $\mathcal{Q}$. In particular, it provides sufficient conditions under which a finite measure $\widehat{\mathbb{P}}$ defined on each $\mathcal{F}_t$ by $\widehat{\mathbb{P}}=M_t.\mathbb{P}$, where $M_t$ is a (true) positive martingale, can be extended to a measure on $\mathcal{F}=\bigvee_{t\geq0}\mathcal{F}_t$.
\begin{lemma} \label{extension}
Let $(\Omega, \mathcal{F}, (\mathcal{F}_t)_{t \geq 0})$ be a filtered probability space satisfying the 
property (P), and let, for $t \geq 0$, $\mathbb{Q}_t$ be a finite measure on $(\Omega, \mathcal{F}_t)$, 
such that for all $t \geq s \geq 0$, $\mathbb{Q}_s$ is the restriction of $\mathbb{Q}_t$ to $\mathcal{F}_s$.
Then, there exists a unique measure $\mathbb{Q}$ on $(\Omega, \mathcal{F})$ such that for all $t \geq 0$,
its restriction to $\mathcal{F}_t$ is equal to $\mathbb{Q}_t$. 
\end{lemma}
\begin{proof} The uniqueness is a direct application of the monotone class theorem, then, let us prove the 
existence. One can obviousely suppose that $\mathbb{Q}_t$ is a probability measure for all $t \geq 0$
(by dividing by $\mathbb{Q}_0 (\Omega)$, if it is different from zero).
We first define, for $t \geq 0$, $\widetilde{\mathbb{Q}}_t$ as the restriction of 
$\mathbb{Q}_t$ to $\mathcal{G}_t$. For $t > s \geq 0$, and for all events $\Lambda_s \in \mathcal{G}_s$, 
one has:
$$\widetilde{\mathbb{Q}}_t (\Lambda_s) = \mathbb{Q}_t (\Lambda_s)= \mathbb{Q}_s (\Lambda_s) = 
\widetilde{\mathbb{Q}}_s (\Lambda_s),$$
 hence, $\widetilde{\mathbb{Q}}_s$ is the restriction of $\widetilde{\mathbb{Q}}_t$ to $\mathcal{G}_s$.
Now, for $t \geq 0$, the map $\pi_t$ is, by assumption, measurable from $(\Omega, \mathcal{G}_t)$ 
to $(\Omega_t, \mathcal{B} (\Omega_t))$ (where $\mathcal{B} (\Omega_t)$ is the Borel $\sigma$-algebra
of $\Omega_t$). Let $\bar{\mathbb{Q}}_t$ be the image of $\widetilde{\mathbb{Q}}_t$ by 
$\pi_t$. By Theorem 1.1.6 of \cite{SV}, for $0 \leq s \leq t$, there exists a conditional probability distribution 
of $\bar{\mathbb{Q}}_t$ given the $\sigma$-algebra $\pi_t(\mathcal{G}_s)$, generated by the 
images, by $\pi_t$, of the sets in $\mathcal{G}_s$. Note that this $\sigma$-algebra
 is included in $\mathcal{B} (\Omega_t)$. Indeed, if $B \in \mathcal{G}_s$, there exists 
$A \in \mathcal{B} (\Omega_t)$ such that $B = \pi_t^{-1} (A)$, and then 
$\pi_t (B) = \pi_t \circ \pi_t^{-1}(A)$, which is equal to $A$ by surjectivity of $\pi_t$. 
 Now, the existence of the conditional probability distribution described above means that one can find
 a family $(Q_{\omega})_{\omega \in \Omega_t}$ of probability measures on $(\Omega_t, \mathcal{B}
(\Omega_t))$ such that:
\begin{itemize}
\item For each $B \in \mathcal{B}(\Omega_t)$, $\omega \rightarrow Q_{\omega} (B)$ is 
$\pi_t (\mathcal{G}_s)$-measurable;
\item For every $A \in \pi_t (\mathcal{G}_s)$, $B \in \mathcal{B} (\Omega_t)$:
$$\bar{\mathbb{Q}}_t (A \cap B) = \int_A Q_{\omega} (B) \, \bar{\mathbb{Q}}_t(d \omega).$$ 
\end{itemize}
\noindent
Now, for all $\omega \in \Omega$, let us define the map $R_{\omega}$ from $\mathcal{G}_t$ to 
$\mathbb{R}_+$, by:
$$R_{\omega} [B] =  Q_{\pi_t (\omega)} [ \pi_t (B)].$$
The map $R_{\omega}$ is a probability measure on $(\Omega, \mathcal{G}_t)$. Indeed, 
$$R_{\omega} [\Omega] = Q_{\pi_t (\omega)}  [\Omega_t] = 1,$$
since $\pi_t$ is surjective. Moreover, let $(B_k)_{k \geq 1}$ be a family of disjoint sets in 
$\mathcal{G}_t$, and $B_0$ their union. By assumption, there exists $(\tilde{B}_k)_{k \geq 0}$ in $\mathcal{B}
 (\Omega_t)$
such that $B_k = \pi_t^{-1} (\tilde{B}_k)$. Since $\pi_t$ is surjective, $\pi_t (B_k) = \tilde{B}_k$. 
Moreover, the sets $(\tilde{B}_k)_{k \geq 1}$ are pairwise disjoint. Indeed, if $x \in  
\tilde{B}_k \cap \tilde{B}_l$ for $k > l \geq 1$, then, by surjectivity, there exists $y \in \Omega$,
such that $x = \pi_t(y)$, which implies $y \in \pi_t^{-1} (\tilde{B}_k) \cap \pi_t^{-1} (\tilde{B}_l)$,
and then $y \in B_k \cap B_l$, which is impossible. Therefore:
\begin{align*}
R_{\omega} [B_0] & =  Q_{\pi_t (\omega)} [ \tilde{B}_0] \\ & = 
\sum_{k \geq 1} Q_{\pi_t (\omega)} [ \tilde{B}_k] \\ & = 
\sum_{k \geq 1} R_{\omega} [B_k] 
\end{align*} 
\noindent
Hence, $R_{\omega}$ is a probability measure. Moreover,
 for each $B \in \mathcal{G}_t$, the map $\omega \rightarrow R_{\omega} (B)$ is the composition
of the measurable maps $\omega \rightarrow  \pi_t(\omega)$ from $(\Omega, \mathcal{G}_s)$
 to $(\Omega_t, \pi_t (\mathcal{G}_s))$, and $\omega' \rightarrow 
Q_{\omega'}[\pi_t(B)]$ from $(\Omega_t, \pi_t (\mathcal{G}_s))$ to $(\mathbb{R}_+, \mathcal{B} (\mathbb{R}_+))$,
and hence, it is $\mathcal{G}_s$-measurable. The measurability of $\pi_t$ follows from the fact that 
for $A \in \mathcal{G}_s$, the inverse image of $\pi_t(A)$ by $\pi_t$ is exactly $A$ (by an assumption
given in the definition of the property (P)). Moreover, for every $A \in  \mathcal{G}_s$, $B \in
\mathcal{G}_t$:
\begin{align*}
\widetilde{\mathbb{Q}}_t (A \cap B) & = \widetilde{\mathbb{Q}}_t [(\pi_t^{-1} \circ \pi_t (A)) 
 \cap (\pi_t^{-1} \circ \pi_t (B))] \\ & =  \widetilde{\mathbb{Q}}_t  [\pi_t^{-1} 
( \pi_t(A) \cap \pi_t(B))] \\ & = \bar{\mathbb{Q}}_t [\pi_t(A) \cap \pi_t(B)]
\\ & = \int_{\Omega_t} \mathds{1}_{\omega \in \pi_t(A)}
Q_{\omega} (\pi_t (B)) \bar{\mathbb{Q}}_t(  d \omega) \\
& = \int_{\Omega} \mathds{1}_{\pi_t(\omega) \in \pi_t(A)} \, Q_{\pi_t(\omega)} (\pi_t (B)) 
\widetilde{\mathbb{Q}}_t (d \omega) \\ & = \int_A R_{\omega} (B) \widetilde{\mathbb{Q}}_t (d \omega)
\end{align*}
\noindent
Finally, we have found a conditional probability distribution of $\widetilde{\mathbb{Q}}_t$
with respect to $\mathcal{G}_s$. Since $\mathcal{G}_s$ is countably generated, this conditional 
probability distribution is regular, by Theorem 1.1.8 in \cite{SV}. One can then apply Theorem
1.1.9, again in \cite{SV}, and since $\mathcal{F}$ is the 
$\sigma$-algebra generated by $\mathcal{G}_t$, $t \geq 0$, 
one obtains a probability distribution $\mathbb{Q}$ on $(\Omega, \mathcal{F})$, such that 
for all integers $n \geq 0$, the restriction of $\mathbb{Q}$ to $\mathcal{G}_n$ is 
$\widetilde{\mathbb{Q}}_n$. Now, for $t \geq 0$, let $\Lambda_t$ be an event in $\mathcal{F}_t$. 
One has, for $n > t$, integer:
$$\mathbb{Q} [\Lambda_t] = \widetilde{\mathbb{Q}}_n [\Lambda_t] = 
\mathbb{Q}_n [\Lambda_t] = \mathbb{Q}_t[\Lambda_t],$$
which implies that $\mathbb{Q}$ satisfies the assumptions of Lemma \ref{extension}.
\end{proof}
Now we give important examples of filtered measurable spaces $(\Omega, \mathcal{F}, (\mathcal{F}_t)_{t \geq 0})$
which satisfy the property (P).
\begin{corollary}
Let 	 $\Omega$ be $\mathcal{C}(\mathbb{R}_+,\mathbb{R}^d)$, the space of continuous functions from 
	$\mathbb{R}_+$ to $\mathbb{R}^d$, or $\mathcal{D}(\mathbb{R}_+,\mathbb{R}^d)$, the space of c\`adl\`ag functions from $\mathbb{R}_+$ 
	to $\mathbb{R}^d$ (for some $d \geq 1$). For $t \geq 0$, define $\mathcal{F}_t = \mathcal{G}_{t+}$,
	where $(\mathcal{G}_{t})_{t \geq 0}$ is the natural filtration of the canonical process $Y$, and $\mathcal{F}=\bigvee_{t\geq0}\mathcal{F}_t$. Then  $(\Omega, \mathcal{F},(\mathcal{F}_t)_{t \geq 0})$ satisfies property (P).
\end{corollary}
\begin{proof}
  Let us prove
this result for c\`adl\`ag functions (for continuous functions, the result is similar and
  is proved in \cite{SV}). For all $t \geq 0$, $\mathcal{G}_t$ is 
generated by the variables $Y_{rt}$, for $r$, rational, in $[0,1]$, hence, it is countably generated.
For the second property, one can take for $\Omega_t$, the set of c\`adl\`ag functions
from $[0,t]$ to $\mathbb{R}^d$, and for $\pi_t$, the restriction to the interval $[0,t]$.
The space $\Omega_t$ is Polish if one endows it with the Skorokhod metric, moreover, 
its Borel $\sigma$-algebra is equal to the $\sigma$-algebra generated by the coordinates, a result from which one easily  deduces the properties of $\pi_t$ which need to be satisfied. The 
third property is easy to check: let us suppose that $(\omega_n)_{n \geq 0}$ is a sequence
 of elements of $\Omega$, such that for all $N \geq 0$,
$$\bigcap_{n = 0}^{N} A_n (\omega_n) \neq \emptyset,$$
where $A_n (\omega_n)$ is the intersection of the sets in $\mathcal{G}_n$ containing $\omega_n$.
Here,  $A_n (\omega_n)$ is the set of functions $\omega'$ which coincide with 
$\omega_n$ on $[0,n]$. Moreover, for $n \leq n'$, integers, the intersection
of $A_n (\omega_n)$ and $A_{n'} (\omega_{n'})$ is not empty, and then $\omega_n$ and $\omega_{n'}$ coincide 
on $[0,n]$. Therefore, there exists a c\`adl\`ag function $\omega$ which coincides
with $\omega_n$ on $[0,n]$, for all $n$, which implies:
$$\bigcap_{n = 0}^{\infty} A_n (\omega_n) \neq \emptyset.$$
\end{proof}
\subsection{The main theorem}
We can now state the main result of the paper:
\begin{theorem} \label{all}
Let $(X_t)_{t \geq 0}$ be a (true) submartingale of the class $(\Sigma)$ (in particular $X_t$ is integrable 
for all $t \geq 0$), defined on a filtered probability space 
$(\Omega, \mathcal{F}, \mathbb{P}, (\mathcal{F}_t)_{t \geq 0})$ which satisfies the property (P).
In particular, $(\mathcal{F}_t)_{t \geq 0}$
 is right-continuous and $\mathcal{F}$ is 
the $\sigma$-algebra generated by $\mathcal{F}_t$, $t \geq 0$. 
Then, there exists a unique $\sigma$-finite measure $\mathcal{Q}$, defined on $(\Omega, \mathcal{F}, \mathbb{P})$ 
such that for $g:= \sup\{t \geq 0, X_t = 0 \}$:
\begin{itemize}
\item $\mathcal{Q} [g = \infty] = 0$;
\item For all $t \geq 0$, and for all $\mathcal{F}_t$-measurable, bounded random variables $F_t$,
$$\mathcal{Q} \left[ F_t \, \mathds{1}_{g \leq t} \right] = \mathbb{E}_{\mathbb{P}} \left[F_t X_t \right].$$
\end{itemize}
\noindent
\end{theorem} 
\begin{remark}
 If $g < \infty$, then $A_{\infty} = A_g < \infty$. Hence, the first condition satisfied by
$\mathcal{Q}$ implies that:
$$\mathcal{Q} [A_{\infty} = \infty] = 0.$$
In other words, $A_{\infty}$ is finite a.s. under $\mathcal{Q}$.
\end{remark}
\begin{proof} Let $f$ be a Borel function from $\mathbb{R}_+$ to $\mathbb{R}_+$, bounded and integrable, 
and let, for $x \geq 0$: $$G(x) := \int_x^{\infty} f(y) \, dy.$$
\noindent
By \cite{N} (Theorem 2.1), one immediately checks that the process $$(M_t^{f} := G(A_t)
 + f(A_t)X_t)
_{t \geq 0},$$ where $A$ is the increasing process of $X$, is a nonnegative local martingale. 
 Moreover, for all $t \geq 0$, if $N$ is the martingale part of
 $X$ and $\mathcal{T}_t$ is the set containing all the stopping times bounded by $t$, then the family
$(N_T)_{T \in \mathcal{T}_t}$ is uniformly integrable (it is included in the set of conditional
expectations of $N_t$, by stopping theorem), and $(A_T)_{T \in \mathcal{T}_t}$ is bounded by $A_t$ ($A$
is increasing), which is integrable (it has the same expectation as $X_t - X_0$). Hence,
$(X_T)_{T \in \mathcal{T}_t}$ is uniformly integrable, which implies, since $f$ and $G$ are uniformly bounded,
that $(M_T^f)_{T \in \mathcal{T}_t}$ is also uniformly integrable. Hence, $M^f$ is a true
martingale. Therefore, by Lemma \ref{extension}, it is possible to construct
a finite measure $\mathcal{P}^f$ on $(\Omega, \mathcal{F}, \mathbb{P})$, uniquely determined by:
$$\mathcal{P}^f [ \Lambda_s] = \mathbb{E}_{\mathds{P}} [\mathds{1}_{\Lambda_s} M_s^f ]$$
for all $s \geq 0$ and for all events $\Lambda_s \in \mathcal{F}_s$. Let us now prove that:
$$\mathcal{P}^f [A_{\infty} = \infty] = 0.$$
Indeed, for $u \geq 0$, let us consider, as in \cite{N}, the right-continuous inverse of $A$: 
$$\tau_{u} := \inf \{t \geq 0, A_t > u \}.$$
\noindent
It is easy to check that for $t, u \geq 0$, the event $\{\tau_{u} \leq t\}$ is equivalent to 
$\{ \forall t' > t, A_{t'} > u\}$, which implies that $\tau_u$ is a stopping time (recall
that $(\mathcal{F}_t)_{t \geq 0}$ is right-continuous). Moreover, if $\tau_{u}< \infty$, then
 $A_{\tau_{u}} = u$ and $X_{\tau_{u}} = 0$. Indeed,
 for all $t > \tau_u$, $A_t > u$, and for all $t < \tau_u$, $A_t \leq u$, which implies the 
first equality by continuity of $A$, for $0 < \tau_u < \infty$ (if $\tau_u = 0$ then 
$u=0$ and the equality is also true). Moreover, if $X_{\tau_u} > 0$, by right-continuity of $X$,
there exists a.s. $\epsilon > 0$ such that $X> 0$ on the interval $[\tau_u , \tau_u + \epsilon]$,
which implies that $A$ is constant on this interval, and then $A_{\tau_u} = A_{\tau_u + \epsilon}
> u$, which is a contradiction.
 Now, for all $t, u \geq 0$,
\begin{align*}
\mathcal{P}^f [A_t > u] & = \mathbb{E}_{\mathbb{P}} \left[ \left(G(A_t) 
 + f(A_t)X_t \right)
 \, \mathds{1}_{A_t > u} \right] \\ & \leq \mathbb{E}_{\mathbb{P}} \left[ \left(G(A_t) 
 + f(A_t)X_t \right)
 \, \mathds{1}_{\tau_{u} \leq t} \right] \\ & = 
\mathbb{E}_{\mathbb{P}} \left[ \left(G(A_{\tau_{u} \wedge t}) + f(A_{\tau_{u} \wedge t})
X_{\tau_{u} \wedge t} \right)
 \, \mathds{1}_{\tau_{u} \leq t} \right] 
\end{align*}
by applying stopping theorem to the stopping time $\tau_{u} \wedge t$. Therefore:
\begin{align*}
 \mathcal{P}^f [A_t > u] & \leq  \mathbb{E}_{\mathbb{P}} \left[ \left(G(A_{\tau_{u}}) + f(A_{\tau_{u}})
X_{\tau_{u}} \right)
 \, \mathds{1}_{\tau_{u} \leq t} \right] \\ & = 
 G(u) \, {\mathbb{P}} \left[ \tau_{u} \leq t \right].
\end{align*}
By taking the increasing limit for $t$ going to infinity, one deduces:
$$ \mathcal{P}^f \left[ \exists t \geq 0, A_t > u \right] \leq G(u) {\mathbb{P}} \left[ \tau_{u} 
< \infty \right].$$
This implies:
$$\mathcal{P}^f \left[ A_{\infty} > u \right] \leq G(u),$$
and by taking $u \rightarrow \infty$,
$$\mathcal{P}^f \left[ A_{\infty} = \infty \right] = 0.$$
 Let us now suppose that $f(x) > 0$ for all $x \geq 0$, and that $G/f$ is uniformly
 bounded on $\mathbb{R}_+$ (for example, one can take $f (x) = e^{-x}$).
Since $\mathcal{P}^f  [A_{\infty} = \infty] = 0$ and $f(A_{\infty}) > 0$, one can define a 
 measure $\mathcal{Q}^{f}$ by the following equality:
$$\mathcal{Q}^{f} [\Lambda] = \mathcal{P}^f \left[ \frac{\mathds{1}_{\Lambda}}{f(A_{\infty}) } \right]$$
for all events $\Lambda \in \mathcal{F}$. This measure is $\sigma$-finite, since for all $\epsilon > 0$:
$$\mathcal{Q}^f [f(A_{\infty}) \geq \epsilon] \leq \frac{1}{\epsilon} \, \mathcal{P}^f (1) < \infty.$$
Now, for $t \geq 0$, and $F_t$, bounded, $\mathcal{F}_t$-measurable:
\begin{align*}
\mathcal{Q}^f  [F_t \, \mathds{1}_{g \leq t}] &  = \mathcal{P}^f \left[ \frac{F_t}{f(A_t)} 
\, \mathds{1}_{g \leq t} 
\right] \\ & = \mathcal{P}^f \left[ \frac{F_t}{f(A_t)} \, \mathds{1}_{d_t = \infty} 
\right] 
\end{align*}
since $A_{\infty} = A_t$ on the event $\{g \leq t\}$, which is equivalent to $\{d_t = \infty\}$,
where $d_t = \inf\{v > t, X_v = 0 \}$. Let us now introduce the filtration 
$(\mathcal{H}_t)_{t \geq 0}$,
where for all $t \geq 0$, $\mathcal{H}_t$ is the $\sigma$-algebra generated by 
$\mathcal{F}_t$ and all the $\mathbb{Q}$-negligible sets of $\mathcal{F}$, where 
$\mathbb{Q} := \mathbb{P} + \mathcal{P}^f$. From \cite{DM}, p. 183, $(\mathcal{H}_t)_{t \geq 0}$ satisfies the usual assumptions and consequently, from the D\'ebut theorem, $d_t$ is an $(\mathcal{H}_t)_{t \geq 0}$-stopping time. From \cite{DM}, Theorem 59, p. 193, there exists an $(\mathcal{F}_t)_{t \geq 0}$-stopping time $\widetilde{d}_t$ such that $d_t=\widetilde{d}_t$, $\mathbb{Q}$-a.s.
One deduces: 
\begin{align*}
 \mathcal{P}^f \left[ \frac{F_t}{f(A_t)} \, \mathds{1}_{d_t \leq u} \right] & =  \mathcal{P}^f \left[ \frac{F_t}{f(A_t)} \, \mathds{1}_{\widetilde{d}_t \leq u} \right] \\ & =  \mathbb{E}_{\mathbb{P}} \left[ \frac{F_t}{f(A_t)} \, M^f_u \, \mathds{1}_{\widetilde{d}_t \leq u} \right].
\end{align*}

By applying stopping theorem to $\tilde{d}_t \wedge u$, one obtains:
$$\mathbb{E}_{\mathbb{P}} [M^f_u | \mathcal{F}_{\widetilde{d}_t \wedge u}] = 
M^f_{\widetilde{d}_t \wedge u}.$$
Hence:
\begin{align*}
 \mathcal{P}^f \left[ \frac{F_t}{f(A_t)} \, \mathds{1}_{d_t \leq u} \right] & 
=  \mathbb{E}_{\mathbb{P}} \left[ \frac{F_t}{f(A_t)} \, M^f_{\widetilde{d}_t} \, \mathds{1}_{\widetilde{d}_t \leq u} \right] \\ & = \mathbb{E}_{\mathbb{P}} \left[ \frac{F_t}{f(A_t)} \, M^f_{d_t} \, \mathds{1}_{d_t \leq u} \right]
\\ & = \mathbb{E}_{\mathbb{P}} \left[ \frac{F_t  G(A_t) }{f(A_t)} \, \mathds{1}_{d_t \leq u} \right]
\end{align*}
\noindent
By taking $u$ going to infinity, one obtains:
\begin{align*}
\mathcal{P}^f \left[ \frac{F_t}{f(A_t)} \, \mathds{1}_{d_t < \infty} \right] 
& = \mathbb{E}_{\mathbb{P}} \left[ \frac{F_t  G(A_t)}{f(A_t)} \mathds{1}_{d_t < \infty} \right]
\end{align*}
Moreover, 
$$\mathcal{P}^f \left[ \frac{F_t}{f(A_t)} \right] 
= \mathbb{E}_{\mathbb{P}} \left[ \frac{F_t  G(A_t)}{f(A_t)} + F_t X_t \right]$$
Therefore, 
\begin{align*}
\mathcal{P}^f \left[ \frac{F_t}{f(A_t)} \, \mathds{1}_{d_t = \infty} \right]  & 
= \mathbb{E}_{\mathbb{P}} [F_t X_t] + 
\mathbb{E}_{\mathbb{P}} \left[ \frac{F_t  G(A_t)}{f(A_t)} \mathds{1}_{d_t = \infty} \right]
\\ & = \mathbb{E}_{\mathbb{P}} [F_t X_t] + 
\mathbb{E}_{\mathbb{P}} \left[ \frac{F_t  G(A_{\infty})}{f(A_{\infty})} \mathds{1}_{d_t = \infty} \right]
\end{align*}
and then:
$$\mathcal{Q}^f [F_t \, \mathds{1}_{g \leq t}] = \mathbb{E}_{\mathbb{P}} [F_t X_t] +
\mathbb{E}_{\mathbb{P}} \left[ \frac{F_t  G(A_{\infty})}{f(A_{\infty})} \mathds{1}_{g \leq t} \right]$$
Now, let us define the measure:
$$\mathcal{P}_1^f := G(A_{\infty}) \,. \mathbb{P}.$$
and the unique measure $\mathcal{P}_2^f$ such that for all $t \geq 0$, its restriction to 
$\mathcal{F}_t$ has density:
$$N_t^f := G(A_t) - \mathbb{E}_{\mathbb{P}} [G(A_{\infty} )| \mathcal{F}_t]  + f(A_t) X_t$$
with respect to $\mathbb{P}$ (note that $N_t^f \geq 0$, $\mathbb{P}$-a.s.). 
It is easy to check that the measures $\mathcal{P}^f$ and $\mathcal{P}_1^f + \mathcal{P}_2^f$ 
have the same restriction to $\mathcal{F}_t$, and by monotone class theorem, they are equal. 
Under $\mathcal{P}_1^f$ and $\mathcal{P}_2^f$, the measure of the event $\{A_{\infty} =\infty\}$ is 
zero, since these two measures are dominated by $\mathcal{P}^f$.
Then, one can define the $\sigma$-finite measures:
$$\mathcal{Q}_1^f := \frac{1}{f(A_{\infty})} \,.\mathcal{P}_1^f$$
and
$$\mathcal{Q}_2^f := \frac{1}{f(A_{\infty})} \,.\mathcal{P}_2^f.$$
The measure $\mathcal{Q}^f$ is the sum  of $\mathcal{Q}_1^f$ and $\mathcal{Q}_2^f$. Now, we 
have:  $$\mathcal{Q}^f_1 [F_t \, \mathds{1}_{g \leq t}] = 
\mathbb{E}_{\mathbb{P}} \left[ \frac{F_t  G(A_{\infty})}{f(A_{\infty})} \mathds{1}_{g \leq t} \right],$$
by using directly the definition of $\mathcal{Q}^f_1$. Moreover, let us recall that:
 $$\mathcal{Q}^f [F_t \, \mathds{1}_{g \leq t}] = \mathbb{E}_{\mathbb{P}} [F_t X_t] +
\mathbb{E}_{\mathbb{P}} \left[ \frac{F_t  G(A_{\infty})}{f(A_{\infty})} \mathds{1}_{g \leq t} \right].$$
In particular, since $G/f$ is assumed to be uniformly bounded:
$$\mathcal{Q}^f [F_t \, \mathds{1}_{g \leq t}] < \infty,$$
This implies that the following equalities are meaningful, and then satisfied, since 
$\mathcal{Q}^f = \mathcal{Q}^f_1 + \mathcal{Q}^f_2$:
\begin{align*} \mathcal{Q}^f_2 [F_t \, \mathds{1}_{g \leq t}] & 
= \mathcal{Q}^f [F_t \, \mathds{1}_{g \leq t}]
- \mathcal{Q}^f_1 [F_t \, \mathds{1}_{g \leq t}] \\ & = 
\left( \mathbb{E}_{\mathbb{P}} [F_t X_t] +
\mathbb{E}_{\mathbb{P}} \left[ \frac{F_t  G(A_{\infty})}{f(A_{\infty})} \mathds{1}_{g \leq t} \right] \right)
\\ & - \mathbb{E}_{\mathbb{P}} \left[ \frac{F_t  G(A_{\infty})}{f(A_{\infty})} \mathds{1}_{g \leq t} \right]
\\ & = \mathbb{E}_{\mathbb{P}}[F_t X_t]
\end{align*}
Hence, the measure $\mathcal{Q}^f_2$ satisfies the second property given in Theorem \ref{all}. By 
applying this property to $F_t = f(A_t)$ (which is bounded, since $f$ is supposed to be bounded)
 and by using the 
fact that $A_t = A_{\infty}$ on $\{g \leq t\}$, one 
deduces:
$$\mathcal{P}_2^f [g \leq t] =  \mathbb{E}_{\mathbb{P}} [f(A_t) X_t]$$
and then (by using the fact that for all $t \geq 0$, $N_t^f$ has an expectation equal to
the total mass of $\mathcal{P}_2^f$):
$$\mathcal{P}_2^f [g > t] = \mathbb{E}_{\mathbb{P}} [G(A_t) - G(A_{\infty})].$$
Since $G(A_t) - G(A_{\infty}) \leq G(0)$ tends $\mathbb{P}$-a.s. to zero when $t$ goes to infinity, 
one obtains:
$$\mathcal{P}_2^f [g = \infty] = 0,$$
and $$\mathcal{Q}_2^f[g = \infty] = 0$$
since $\mathcal{Q}_2^f$ is absolutely continuous with respect to $\mathcal{P}_2^f$. 
Therefore, the measure $\mathcal{Q}$ exists: let us now prove its uniqueness (which implies,
in particular, that $\mathcal{Q}_2^f$ is in fact independent of the choice of $f$). 
If $\mathcal{Q}'$ and $\mathcal{Q}''$ satisfy the conditions defining $\mathcal{Q}$, one has, 
for all $t \geq 0$ and all events $\Lambda_t \in \mathcal{F}_t$:
 $$\mathcal{Q}' [\Lambda_t, g \leq t] = \mathcal{Q}'' [\Lambda_t, g \leq t]$$
Now let $u > t \geq 0$, and let $\Lambda_u$ be in $\mathcal{F}_u$.
One can check that:
 $$\mathcal{Q}' [\Lambda_u, g \leq t] = \mathcal{Q}' [\Lambda_t, d_t > u, g \leq u]$$
If $\mathcal{H}'_t$ is the $\sigma$-algebra generated by $\mathcal{F}_t$ and 
the $(\mathcal{Q}' + \mathcal{Q}'')$-negligible sets of $\mathcal{F}$, then $d_t$ is 
a stopping time with respect to the right-continuous filtration $(\mathcal{H}'_t)_{t \geq 0}$.
Hence the event $\Lambda'_u := \Lambda_t \cap \{d_t > u\}$ is in $\mathcal{H}'_u$, 
and then, there exists an event $\Lambda''_u \in \mathcal{F}_u$ such that 
$$\mathcal{Q}' [(\Lambda''_u \backslash \Lambda'_u) \cup (\Lambda'_u \backslash \Lambda''_u)] = 0$$
and 
 $$\mathcal{Q}'' [(\Lambda''_u \backslash \Lambda'_u) \cup (\Lambda'_u \backslash \Lambda''_u)] = 0.$$
One then deduces that:
\begin{align*}
\mathcal{Q}' [\Lambda_u, g \leq t] & = \mathcal{Q}' [\Lambda'_u, g \leq u] \\ 
& = \mathcal{Q}' [\Lambda''_u, g \leq u] 
\\ & = \mathcal{Q}'' [\Lambda''_u, g \leq u] 
\\ & = \mathcal{Q}'' [\Lambda'_u, g \leq u] 
\\ & = \mathcal{Q}''[\Lambda_u, g \leq t].
\end{align*}
By the monotone class theorem, applied to the restrictions of $\mathcal{Q}'$ and $\mathcal{Q}''$
to the set $\{g \leq t \}$, one has for all $\Lambda \in \mathcal{F}$:
$$\mathcal{Q}'[\Lambda, g \leq t]  = \mathcal{Q}'' [\Lambda, g \leq t].$$
By taking the increasing limit for $t$ going to infinity, 
$$\mathcal{Q}'[\Lambda, g < \infty]= \mathcal{Q}'' [\Lambda, g < \infty].$$
Now, by assumption:
$$\mathcal{Q}' [g= \infty] = \mathcal{Q}'' [g= \infty] = 0,$$
which implies:
$$\mathcal{Q}' [\Lambda] = \mathcal{Q}'' [\Lambda]. $$
This completes the proof of Theorem \ref{all}.
\end{proof}
A careful look at the proof of Theorem \ref{all} shows that the result is valid if $t$ is replaced by a bounded stopping time $T$. Moreover, for submartingales of the class $(\Sigma)$ which are also of class $(D)$, we can take a filtration $(\mathcal{F}_t)$ which satisfies the usual assumptions. More precisely, the following result holds:
\begin{corollary}\label{cor1}
Let $(X_t)_{t \geq 0}$ be a submartingale of the class $(\Sigma)$.
\begin{enumerate}
\item If the  filtered probability space 
$(\Omega, \mathcal{F}, \mathbb{P}, (\mathcal{F}_t)_{t \geq 0})$  satisfies the property (P), then there exists a unique $\sigma$-finite measure $\mathcal{Q}$, defined on $(\Omega, \mathcal{F}, \mathbb{P})$ 
such that for $g:= \sup\{t \geq 0, X_t = 0 \}$:
\begin{itemize}
\item $\mathcal{Q} [g = \infty] = 0$:
\item For any bounded stopping time $T$, and for all $\mathcal{F}_T$-measurable, bounded random variables $F_T$,
$$\mathcal{Q} \left[ F_T \, \mathds{1}_{g \leq T} \right] = \mathbb{E}_{\mathbb{P}} \left[F_T X_T \right].$$
\end{itemize}
\item if the $X$ is of class $(D)$ and the  filtered probability space 
$(\Omega, \mathcal{F}, \mathbb{P}, (\mathcal{F}_t)_{t \geq 0})$  satisfies the usual assumptions or the property (P), then for any stopping time $T$
$$X_T=\mathbb{E}[X_\infty \mathds{1}_{g\leq T}|\mathcal{F}_T],$$where as usual $g:= \sup\{t \geq 0, X_t = 0 \}$.
\end{enumerate}
\end{corollary}
\begin{remark}
Part $(2)$ of Corollary \ref{cor1}, under the usual assumptions, is proved in \cite{CNP}. 
\end{remark}
Let us note that, in the proof of Theorem \ref{all}, if $f$ does not vanish, is bounded 
and if $G/f$ is also bounded then 
the finite measure $\mathcal{P}_2^f$ has density $f(A_{\infty})$ with respect to $\mathcal{Q}$.
Now, one can prove that, in fact, these conditions on $f$ are not needed. More precisely, one
 has the following:
\begin{proposition} \label{f}
Let $f$ be an integrable function from $\mathbb{R}_+$ to $\mathbb{R}_+$.
Then, there exists a unique finite (positive) measure $\mathcal{M}^f$ such that:
$$\mathcal{M}^f[F_t] = \mathbb{E}_{\mathbb{P}} [F_t N^f_t]$$
for all $t \geq 0$, and for all bounded, $\mathcal{F}_t$-measurable functionals $F_t$,
where the process $(N_t^f)_{t \geq 0}$ is given by:
$$N_t^f := G(A_t) - \mathbb{E}_{\mathbb{P}} [G(A_{\infty})| \mathcal{F}_t] + f(A_t)X_t$$
for $$G(x) := \int_x^{\infty} f(y) dy.$$
In particular, $(N_t^f)_{t \geq 0}$ is a nonnegative martingale. Moreover,
 the measure $\mathcal{M}^f$ is absolutely continuous with respect to $\mathcal{Q}$, with 
density $f(A_{\infty})$. 
\end{proposition}
\begin{proof}
In the proof of Theorem \ref{all}, we have shown this result if $f$ is strictly positive, 
bounded, and if $G/f$ is also bounded (recall that $G(x)$ is the integral of $f$ between $x$ and
infinity). One can now prove Proposition \ref{f} for
any measurable, bounded, nonnegative functions $f$ with compact support. Indeed, if $f$ is such
a function, one can find $f_1$ and
$f_2$, bounded, strictly positive, integrable, such that, with obvious notation, $G_1/f_1$ and $G_2/f_2$ are
bounded, and $f = f_1 - f_2$
 (for example, one can take $f_1(x) := f(x) + e^{-x}$ and $f_2 (x) := e^{-x}$). One has,
for all $t \geq 0$, and for all bounded, $\mathcal{F}_t$-measurable random variables $F_t$:
$$\mathcal{M}^{f_1}[F_t] = \mathbb{E}_{\mathbb{P}} [F_t N^{f_1}_t],$$
and 
$$\mathcal{M}^{f_2}[F_t] = \mathbb{E}_{\mathbb{P}} [F_t N^{f_2}_t].$$
Now $N^f$ is the difference of $N^{f_1}$ and $N^{f_2}$, and then, it is a (nonnegative) martingale. 
Hence, there exists a unique finite measure $\mathcal{M}$ such that: 
$$\mathcal{M} [F_t] = \mathbb{E}_{\mathbb{P}} [F_t N^f_t].$$
Therefore $\mathcal{M}^f$ exists, is unique, and since $\mathcal{M}^{f_1}$
 and $\mathcal{M}^{f_2}+ \mathcal{M}^f$ coincide
on $\mathcal{F}_t$ for all $t \geq 0$:
$$\mathcal{M}^{f}[F_t] = \mathcal{M}^{f_1} [F_t] - \mathcal{M}^{f_2}[F_t]$$
(this equality is meaningful because all the measures involved here are finite).
Since Proposition \ref{f} is satisfied for $f_1$ and $f_2$:
$$\mathcal{M}^f [F_t] = \mathcal{Q} [F_t f_1(A_{\infty})] - \mathcal{Q} [F_t f_2(A_{\infty})],$$
which implies
$$\mathcal{M}^f [F_t]  = \mathcal{Q} [F_t f(A_{\infty})].$$
By monotone class theorem, $f$ satisfies Proposition \ref{f}. Now, let us only suppose that $f$ is 
nonnegative and integrable. There exists nonnegative, measurable, bounded functions $(f_k)_{k \geq 1}$ with
compact support, such that:
$$f = \sum_{k \geq 1} f_k.$$
With obvious notation, one has:
$$G = \sum_{k \geq 1} G_k,$$
and then, for all $t \geq 0$:
$$G(A_t) = \sum_{k \geq 1} G_k(A_t)$$
and
$$\mathbb{E}_{\mathbb{P}} [G(A_{\infty})| \mathcal{F}_t] = \sum_{k \geq 1} 
\mathbb{E}_{\mathbb{P}} [G_k(A_{\infty})| \mathcal{F}_t],$$
$\mathbb{P}$-a.s., where the two sums are uniformly bounded by $G(0)$. This boundedness implies that
one can substract the second sum from the first, and obtain:
$$ G(A_t) - \mathbb{E}_{\mathbb{P}} [G(A_{\infty})| \mathcal{F}_t] = \sum_{k \geq 1} \left(G_k(A_t) 
- \mathbb{E}_{\mathbb{P}} [G_k(A_{\infty})| \mathcal{F}_t] \right)$$
almost surely. Moreover: 
$$f(A_t)X_t = \sum_{k \geq 1} f_k(A_t) X_t,$$
and then, $\mathbb{P}$-a.s.:
$$N_t^f = \sum_{k \geq 1} N_t^{f_k}.$$
We know that $\mathcal{M}^{f_k}$ is well-defined for all $k \geq 1$, hence, one can consider the 
measure:
$$\mathcal{M} := \sum_{k \geq 1} \mathcal{M}^{f_k}.$$
Now, for $t \geq 0$ and $F_t$, bounded, $\mathcal{F}_t$-measurable:
\begin{align*}
\mathcal{M} [F_t] & = \sum_{k \geq 1} \mathcal{M}^{f_k} [F_t] \\
& = \sum_{k \geq 1} \mathbb{E}_{\mathbb{P}} [F_t N^{f_k}_t] \\
& = \mathbb{E}_{\mathbb{P}} [F_t N^f_t].
\end{align*}
Hence, the measure $\mathcal{M}^f$ is well-defined, unique by monotone class theorem, and 
is equal to $\mathcal{M}$. 
Now, one has, for all $k \geq 1$:
$$\mathcal{M}^{f_k} = f_k(A_{\infty}) \, . \mathcal{Q}.$$
Since $\mathcal{M}^f$ is the sum of the measures $\mathcal{M}^{f_k}$,
$$\mathcal{M}^{f} = \left[ \sum_{k \geq 1} f_k (A_{\infty}) \right] \,. \mathcal{Q} = f(A_{\infty}). \mathcal{Q}$$
which completes the proof of Proposition \ref{f}.
\end{proof}
Another question which is quite natural to ask is the following: since $\mathcal{Q} [A_{\infty} = \infty] = 0$,
what is the image of $\mathcal{Q}$ by the functional $A_{\infty}$ (in other words, what is the 
"distribution of $A_{\infty}$ under $\mathcal{Q}$")? This question can be solved in any 
case:
\begin{proposition} \label{Ainfty}
Let $(X_t)_{t \geq 0}$ be a submartingale of class $(\Sigma)$, which satisfies all the 
assumptions of Theorem \ref{all}. Then, if $(A_t)_{t \geq 0}$ is 
the increasing process of $(X_t)_{t \geq 0}$, the image by the functional $A_{\infty}$ of
the measure $\mathcal{Q}$ defined in Theorem \ref{all}, is a measure on $\mathbb{R}_+$, 
equal to the sum of $\mathbb{E}_{\mathbb{P}} [X_0]$ times Dirac measure at zero, and another measure, 
absolutely continuous with respect to Lebesgue measure, with density $\mathbb{P} [A_{\infty} > u]$
at any $u \in \mathbb{R}_+$. In particular, if $A_\infty=\infty$, $\mathbb{P}$-almost surely, then
 this image measure is $\mathbb{E}_{\mathbb{P}} [X_0]\delta_{0}+\mathds{1}_{\mathbb{R}_+}\lambda$, where
 $\lambda$ is Lebesgue measure on $\mathbb{R}_+$, and $\delta_0$ is Dirac measure at zero.
\end{proposition}
\begin{proof}
Let $f$ be an integrable function from $\mathbb{R}_+$ to $\mathbb{R}_+$. By taking 
the notation of Proposition \ref{f}, one has:
 $$\mathcal{M}^f = f(A_{\infty})\, .\mathcal{Q}.$$
Therefore, $\mathcal{Q} [f(A_{\infty})]$ is the total mass of $\mathcal{M}^f$, and then, the expectation
of: $$N^f_0 = G(0) - \mathbb{E}_{\mathbb{P}} [G(A_{\infty})| \mathcal{F}_0 ] + f(0)X_0$$
By applying this result to $f = \mathds{1}_{[0,u]}$, one deduces, for any $u \geq 0$:
\begin{align*}
\mathcal{Q} [A_{\infty} \leq u] & = u - \mathbb{E}_{\mathbb{P}} [ (u-A_{\infty})_+ ] + f(0) \mathbb{E}_{\mathbb{P}} [X_0] 
\\ & = \mathbb{E}_{\mathbb{P}} [A_{\infty} \wedge u] + f(0) \mathbb{E}_{\mathbb{P}} [X_0] 
\\ & = \int_0^u \mathbb{P} [A_{\infty} > v] \, dv + f(0) \mathbb{E}_{\mathbb{P}}[X_0].
\end{align*}
\noindent
which implies Proposition \ref{Ainfty}.
\end{proof} 
\begin{remark}
When $X$ is also of class $(D)$,  $ \mathbb{P} [A_{\infty} > v] $ is computed in  \cite{N}, Theorem 4.1.
\end{remark}
\subsection{An extension of Doob's optional stopping theorem}
We shall now see how Theorem \ref{all} and Corollary \ref{cor1} can be interpreted as an extension of Doob's optional theorem to continuous martingales which are not necessarily uniformly integrable on the one hand, and to the larger class of processes of the class $(\Sigma)$.

Let $M$ be a continuous martingale; then $M^+$ and $M^-$ are both of class $(\Sigma)$. If $g=\sup\{t \geq 0:\;M_t=0\}$, then under the assumptions of Theorem \ref{all}, there exist two $\sigma$-finite measures $\mathcal{Q}^{(+)}$ and $\mathcal{Q}^{(-)}$ such that 
\begin{itemize}
\item $\mathcal{Q}^{(\pm)} [g = \infty] = 0$:
\item For all $t \geq 0$, and for all $\mathcal{F}_t$-measurable, bounded random variables $F_t$,
$$\mathcal{Q}^{(\pm)} \left[ F_t \, \mathds{1}_{g \leq t} \right] = \mathbb{E}_{\mathbb{P}} \left[F_t M_t^{\pm} \right].$$
\end{itemize}
Now since $M=M^{+}-M^{-}$, we deduce from Theorem \ref{all} and Corollary \ref{cor1} the following solution to Problem 2:
\begin{proposition} \label{Doob}
\item Let $M$ be a continuous martingale defined on a filtered probability space 
$(\Omega, \mathcal{F}, \mathbb{P}, (\mathcal{F}_t)_{t \geq 0})$ which satisfies the property (P).
 Then there exist two $\sigma$-finite measures $\mathcal{Q}^{(+)}$ and $\mathcal{Q}^{(-)}$, such that for 
any bounded stopping time $T$ and any bounded $\mathcal{F}_T$-measurable variable $F_T$,
 $$\left(\mathcal{Q}^{(+)}-\mathcal{Q}^{(-)}\right) \left[ F_T\, \mathds{1}_{g \leq T} \right] = \mathbb{E} \left[F_T M_T \right],$$
with $g=\sup\{t \geq 0 :\;M_t=0\}$. The measures $\mathcal{Q}^{(+)}$ and $\mathcal{Q}^{(-)}$ are obtained by applying
Theorem \ref{all} to the submartingales $M^+$ and $M^-$. 
\end{proposition}
\begin{remark}
If the martingale $M$ is uniformly integrable, then following Corollary \ref{cor1}, one can 
also work with a filtration satisfying the usual assumptions and take any
 stopping time $T$, not necessarily bounded. Consequently, Proposition \ref{Doob} can be viewed as an extension of Doob's optional stopping theorem: the terminal value of the martingale $M$ has to be replaced by $\left(\mathcal{Q}^{(+)}-\mathcal{Q}^{(-)}\right)$ which is a signed measure when restricted to the sets $\mathds{1}_{g\leq t}$. Theorem \ref{all} and Corollary \ref{cor1} can  in turn be interpreted as an extension of the stopping theorem to the larger class of submartingales of the class $(\Sigma)$.  
\end{remark}
\section{Some examples}
Now, let us study in more details several consequences of Theorem \ref{all}. 
\subsection{The case of a the absolute value, or the positive part, of a martingale} We suppose
that $X_t = M_t^+$,
$X_t = M_t^-$ or $X_t = |M_t|$, where $(M_t)_{t \geq 0}$ is a continuous martingale. In this case, 
$X$ is a submartingale of class $(\Sigma)$, and its increasing process is half of 
the local time of $M$ in the two first cases, and the local time of $M$ in the third case. Therefore,
one can apply Theorem \ref{all}. In particular, if $(X_t)_{t \geq 0}$ is a strictly positive
martingale, then it is a submartingale of class $(\Sigma)$, with increasing process identically
equal to zero. One deduces that for any nonnegative, integrable function $f$, 
$N^f_t = f(0) X_t$, which implies that for all $t \geq 0$, the restriction of $\mathcal{M}^f$ to 
$\mathcal{F}_t$ has density 
$f(0) X_t$ with respect to $\mathbb{P}$. Hence, since $f(A_{\infty}) = f(0)$, the restriction 
of $\mathcal{Q}$ to $\mathcal{F}_t$ has density $X_t$ with respect to $\mathbb{P}$. In 
particular, $\mathcal{Q}$ is a finite mesure, and $X$ does not vanish under $\mathcal{Q}$,
i.e. $$\mathcal{Q} [\exists t \geq 0, X_t = 0]=0.$$
\noindent

\subsection{The case of the draw-down of a martingale} Let $(M_t)_{t \geq 0}$ be a c\`adl\`ag martingale,
 starting at zero, without positive jumps.
This assumption implies that its supremum 
$$S_t := \underset{s \leq t}{\sup} \, M_s$$
is a.s. continuous with respect to $t$. The process
$$(X_t := S_t - M_t)_{t \geq 0}$$ is then a submartingale of class $(\Sigma)$
with martingale part $-M$ and increasing process $S$. One obtains, for all 
$t \geq 0$ and $F_t$ bounded, $\mathcal{F}_t$-measurable:
$$\mathcal{Q} \left[ F_t \, \mathds{1}_{g \leq t} \right] = \mathbb{E}_{\mathbb{P}} \left[F_t (S_t - M_t) \right]$$
where, in this case, $g$ is the last time when $M$ reaches its overall supremum. 

\subsection{The uniformly integrable case} Let us suppose that, in Theorem \ref{all}, the family of
variables $(X_t)_{t \geq 0}$ is uniformly integrable. In this case, $(\mathbb{E}_{\mathbb{P}} [X_t])_{t \geq 0}$,
and then $(\mathbb{E}_{\mathbb{P}} [A_t])_{t \geq 0}$ are uniformly bounded. By monotone convergence,
$A_{\infty}$ is integrable, and in particular finite a.s. Since $(A_t)_{t \geq 0}$ and $(X_t)_{t \geq 0}$ are 
uniformly integrable, $(N_t)_{t \geq 0}$ is a uniformly integrable martingale, which implies that 
there exists $N_{\infty}$ such that for all $t \geq 0$, $N_{t} = \mathbb{E} [N_{\infty} | \mathcal{F}_t]$
and $N_t$ tends a.s. to $N_{\infty}$ for $t$ going to infinity. One deduces that $X_t$ tends a.s. 
to $X_{\infty}:=N_{\infty} + A_{\infty}$. Now, for all nonnegative, bounded, integrable functions $f$, the
 martingale $N^f$ is uniformly inegrable. Moreover, if 
$f$ is continuous, $G(A_t) + X_t f(A_t)$ tends a.s. to $G(A_{\infty}) + X_{\infty}
f(A_{\infty})$ when $t \rightarrow \infty$, and the uniformly integrable martingale 
$(\mathbb{E}[G(A_{\infty})| \mathcal{F}_t])_{t \geq 0}$ tends a.s. to $G(A_{\infty})$. Therefore, the 
terminal value of $N^f$ is $X_{\infty} f(A_{\infty})$, which implies that $\mathcal{M}^f$
has density $X_{\infty} f(A_{\infty})$ with respect to $\mathbb{P}$, 
and finally: $$\mathcal{Q} = X_{\infty} \,. \mathbb{P}.$$
 This case was essentially obtained by Az\'ema, 
Meyer and Yor in \cite{AMY} and  in \cite{CNP} in relation with problems from mathematical finance. The particular case where 
 $X_t = |M_t|$, where $(M_t)_{t \geq 0}$ is a continuous uniformly integrable
 martingale, starting at zero, and for which the measure $\mathcal{Q}$ has density $|M_{\infty}|$
with respect to $\mathbb{P}$, was studied in \cite{AY2}, \cite{AY1}. 

\subsection{The case where $A_{\infty}$ is infinite almost surely} In this case, for 
any nonnegative, integrable function $f$, one has: 
$$ N_t^f = G(A_t) + f(A_t) X_t.$$
Moreover, if $X_0 = 0$ a.s., then the image of $\mathcal{Q}$ by $A_{\infty}$ is simply Lebesgue measure. 
There are several interesting examples of this particular case. 

1) It $X_t = M_t^+$,
$X_t = M_t^-$ or $X_t = |M_t|$, where $M$ is a continuous martingale, 
then we are in the case: $A_{\infty} = \infty$ iff the total local time of $M$ is a.s. infinite, or, 
equivalently, iff the overall supremum of $|M|$ is a.s. infinite. 
This condition is satisfied, in particular, if $M$ is a Brownian motion. More precisely, let us suppose that 
$\Omega = \mathcal{C}  
(\mathbb{R}_+, \mathbb{R})$, $(\mathcal{F}_t)_{t \geq 0}$ is 
the smallest right-continuous filtration, containing the natural filtration of the
 canonical process $(Y_t)_{t \geq 0}$, 
and $\mathbb{P}$ is Wiener measure. If $X_t = |Y_t|$, 
$X_t = Y_t^+$ or $X_t = Y_t^-$, the $\sigma$-finite measure 
$\mathcal{Q}$ described in Theorem \ref{all} was already
studied in  \cite{NRY}, Chapter 1. This measure satisfies a slightly more general 
result than what is written in Theorem \ref{all}.
Indeed, in their monograph, Najnudel, Roynette and Yor prove that 
there exists a unique $\sigma$-finite measure $\mathcal{W}$ on 
$\Omega$ such that for all $t \geq 0$, for all bounded, $\mathcal{F}_t$-measurable
functionals $F_t$, and for all $a \in \mathbb{R}$:
$$\mathcal{W} [F_t \mathds{1}_{g_a \leq t}] = \mathbb{P} [F_t \, |Y_t - a|],$$
$$\mathcal{W} [g_a = \infty] = 0$$
where $$g_a := \sup\{t \geq 0, Y_t = a \}.$$
Moreover $\mathcal{W}$ can be decomposed (in unique way) as the sum of 
two $\sigma$-finite measures $\mathcal{W}^+$ and $\mathcal{W}^-$, such
that:
$$\mathcal{W}^+ [F_t \mathds{1}_{g_a \leq t}] = \mathbb{P} [F_t \,
 (Y_t - a)^+],$$
$$\mathcal{W}^- [F_t \mathds{1}_{g_a \leq t}] = \mathbb{P} [F_t \,
 (Y_t - a)^-],$$
$$\mathcal{W}^+ [E_-] = \mathcal{W}^- [E_+] = 0$$
where $E_-$ is the set of trajectories which do not tend to $+ \infty$, 
and $E_+$ is the set of trajectories which do not tend to $- \infty$.
With these definitions, the measure $\mathcal{Q}$ is equal to $\mathcal{W}^+$
if $X_t = Y_t^+$, $\mathcal{W}^-$ if $X_t = Y_t^-$ and $\mathcal{W}$ if
$X_t = |Y_t|$. 

2) Let $(M_t)_{t \geq 0}$ be a c\`adl\`ag martingale, starting at zero, 
 without positive jumps. The process
$$(X_t := S_t - M_t)_{t \geq 0}$$ is a submartingale of class $(\Sigma)$
with martingale part $-M$ and increasing process $S$, and one has $A_{\infty} = \infty$ a.s.,
iff the overall supremum of $M$ is a.s. infinite. A particular case where this condition holds 
is, again, when $M$ is a Brownian motion. More precisely, if one takes the same filtered probability space as in 
the previous example, and if $X_t = ( \underset{s \leq t}{\sup} \, Y_s) -
Y_t$, then the $\sigma$-finite measure exists and is in fact equal to
$\mathcal{W}^-$. Note that the image of this measure by $X$ is equal to the image of $\mathcal{W}$ by
the absolute value. 

3) Another interesting example is studied in Chapter 3 of \cite{NRY}.
Let us take the same filtered measurable space as in the previous examples, 
endowed with a probability measure $\mathbb{P}$ under which the canonical 
process  $(Y_t)_{t \geq 0}$ is a recurrent, homogeneous diffusion with values in $\mathbb{R}_+$, starting at 
zero, and such that
 zero is an instantaneously reflecting barrier. We suppose that the infinitesimal 
generator $\mathcal{G}$ of $Y$ satisfies (for $x \geq 0$):
$$\mathcal{G} f (x) = \, \frac{d}{dm} \, \frac{d}{dS} \, f(x)$$
where $S$ is a continuous, strictly increasing function such that $S(0) = 0$ and $S(\infty) = \infty$, 
and $m$ is the speed measure, satisfying $m(\{0\}) = 0$. There exists a jointly continuous family $(L_t^y)_
{t, y \geq 0}$ of local times of $Y$, satisfying: $$\int_0^t h(Y_s) \, ds = \int_0^{\infty} h(y)
L_t^y \, m(dy)$$
for all borelian functions $h$ from $\mathbb{R}_+$ to $\mathbb{R}_+$. 
If we define the process $(X_t)_{t \geq 0}$ by:
$$X_t = S(Y_t)$$
then $(X_t - L_t^0)_{t \geq 0}$ is a $(\mathcal{F}_t)_{t \geq 0}$-martingale. Hence, if $\mathcal{F}$ is
 the $\sigma$-algebra generated by $(\mathcal{F}_t)_{t \geq 0}$, 
the assumptions of Theorem \ref{all} 
 are satisfied, and $L_{\infty}^0$ is infinite, since the diffusion $Y$ is recurrent. 
The $\sigma$-finite measure $\mathcal{Q}$ is given by the formula:
$$\mathcal{Q} = \int_0^{\infty} dl \, \mathbb{Q}_l,$$
where $\mathbb{Q}_l$ is the law of a process $(Z^{l}_t)_{t \geq 0}$, defined in the following way: let 
$\tau_l$ be the inverse local time at $l$ (and level zero) of a diffusion $R$, which has a law equal
 to the distribution of $Y$ under $\mathbb{P}$, 
and let $(\tilde{R}_u)_{u \geq 0}$ be an homogeneous diffusion, independent of $R$, starting at zero, never
hitting zero again, and 
such that for $0 \leq u < v$, $x, y > 0$:
$$P[\tilde{R}_v \in dy \, | \tilde{R}_u = x] = \frac{S(y)}{S(x)} P [R_v \in dy, \forall 
w \in [u,v], R_w > 0 \, |R_u = x]$$
(intuitively, the law of $(\tilde{R}_u)_{u \geq 0}$ is the law of $(R_u)_{u \geq 0}$, conditioned not to 
vanish), then $Z^l$ satisfies
$$Z^l_t = R_t$$
for $t \leq \tau_l$, and 
$$Z^l_{\tau_l + u}  = \tilde{R}_u$$
for $u \geq 0$. Theorem \ref{all} applies, in particular, if $Y$ is a Bessel process of dimension 
$d \in (0,2)$. If $d = 2(1- \alpha)$ (which imples $0 < \alpha < 1$), one obtains:
$$X_t = (Y_t)^{2 \alpha}  = (Y_t)^{2 - d}.$$
In this case, the process $(\tilde{R}_u)_{u \geq 0}$, involved in an essential way in
 the construction of $\mathcal{Q}$, is a Bessel process of dimension $4-d= 2 (1 + \alpha)$.
For $d = 1$ ($\alpha = 1/2$), $(X_t = Y_t)_{t \geq 0}$ is the absolute value of a Brownian motion, and
$\tilde{R}$ is a Bessel process of dimension 3. 

4) Let $\Omega$ be the space of continous functions from $\mathbb{R}_+$ to $\mathbb{R}$,
$(\mathcal{H}_t)_{t \geq 0}$
the smallest right-continous filtration, containing the natural filtration of the canonical
 process $(Y_t)_{t \geq 0}$, $\mathcal{H}$ the $\sigma$-algebra generated by $(\mathcal{H}_t)_{t \geq 0}$
and $\mathbb{P}$ the probability measure under which $(Y_t)_{t \geq 0}$ is a Bessel process of 
dimension $d := 2(1- \alpha)$ for $0 < \alpha < 1$. For $t \geq 0$, let us
 take the notation: $$g(t) := \sup \{u \leq t, 
Y_u = 0 \}$$ and let $(\mathcal{F}_t)_{t \geq 0}$ be the filtration of
the zeros of $Y$, i.e. $\mathcal{F}_t = \mathcal{H}_{g(t)}$. One defines the
 $\sigma$-algebra $\mathcal{F}$ as the $\sigma$-algebra
generated by $(\mathcal{F}_t)_{t \geq 0}$, i.e. by the zeros of $Y$. Now, the process
$$(X_t := (t-g(t))^{\alpha})_{t \geq 0}$$ is a $(\mathcal{F}_t)_{t \geq 0}$-submartingale of 
class $(\Sigma)$, and its increasing process $(A_t)_{t \geq 0}$ is given by: $$A_t = 
\frac{1}{2^{\alpha} \Gamma(1+\alpha)} \, L_t(Y)$$
where $L_t(Y)$ is the local time of $Y$ at zero, defined as the increasing process of the 
submartingale $(Y_t^{2\alpha})_{t \geq 0}$, which is of class $(\Sigma)$ (see \cite{N}, and the previous example). 
Since $Y$ is recurrent, $A_{\infty} = \infty$ a.s. Now, let $\mathcal{R}$ be the $\sigma$-finite
 measure on $(\Omega,
 (\mathcal{H}_t)_{t \geq 0}, \mathcal{H})$ which is equal to the measure 
$\mathcal{Q}$ of example 3). Because of this example, one has, 
for all bounded, $\mathcal{H}_t$-measurable functions $F_t$:
$$\mathcal{R}[F_t \, \mathds{1}_{g \leq t}] = \mathbb{E}_{\mathbb{P}} 
[F_t Y_t^{2\alpha}]$$
where $g$ is the last zero of $Y$, equal to the last zero of $X$. Now, if 
$F_t$ is $\mathcal{F}_t$-measurable, then one obtains
$$\mathcal{R}[F_t \, \mathds{1}_{g \leq t}] = \mathbb{E}_{\mathbb{P}} 
[F_t  \mathbb{E}_{\mathbb{P}}[ Y_t^{2\alpha}| \mathcal{F}_t]]$$
which implies:
$$\mathcal{R}[F_t \, \mathds{1}_{g \leq t}] = 2^{\alpha} \Gamma (1+\alpha) \,
\mathbb{E}_{\mathbb{P}} [F_t X_t]$$
Therefore, the measure $\mathcal{Q}$ satisfying the conditions given in Theorem \ref{all} 
is the restriction of the measure $$\tilde{\mathcal{Q}} := \frac{1}{2^{\alpha} \Gamma (1+\alpha)} \, \mathcal{R}$$
 to the $\sigma$-algebra $\mathcal{F}$ (generated by the zeros of $Y$). Moreover, 
the image of $\mathcal{Q}$ by $X$ is: $$ \mathcal{S} := \frac{1}{2^{\alpha} \Gamma (1+\alpha)} \, 
\int_0^{\infty} dl \, \mathbb{S}_l$$ where 
 $\mathbb{S}_l$ is the law of a process $(V^{l}_t)_{t \geq 0}$, defined in the following way: let 
$\tau_l$ be the inverse local time at $l$ (and level zero) of a diffusion $R$, with the same law as 
$Y$ under $\mathbb{P}$, and let $\gamma(t)$ be the last zero of $R$ before time $t$, for all $t \geq 0$,
 $V^l$ satisfies
$$V^l_t = (t-g(t))^{\alpha}$$
for $t \leq \tau_l$, and 
$$V^l_{\tau_l + u}  = u^{\alpha}$$
for $u \geq 0$. 
Note that in this case, we have not checked that the filtered probability space has property (P). However, 
we have proved that the conclusion of Theorem \ref{all} holds in this case.

\end{document}